\documentclass[a4paper,10pt]{amsart}
\usepackage{ae,amsfonts,euscript,enumerate}
\usepackage{amsmath,euscript}
\usepackage{amssymb}
\usepackage{amsthm}
\usepackage{enumerate}
\usepackage{microtype}
\usepackage{amsfonts}
\usepackage{comment}
\usepackage{colortbl}
\usepackage{amssymb,amsmath,array}

\theoremstyle{plain}
\newtheorem{theorem}{Theorem}[section]
\newtheorem{lemma}[theorem]{Lemma}
\newtheorem{proposition}[theorem]{Proposition}

\newcommand{\B}{\mathbb}
\newcommand{\C}{\mathcal}

\newcommand{\ga}{\alpha}

\begin{document}

\title[Addendum]{Addendum to: On the rational approximation of the sum of the reciprocals of the Fermat numbers}
\date{\today}
\author{Michael Coons}
\address{School of Mathematical and Physical Sciences\\
University of Newcastle\\
Australia}
\email{Michael.Coons@newcastle.edu.au}
\thanks{The research of M.~Coons was supported by ARC grant DE140100223.}

\begin{abstract}
As a corollary of the main result of our recent paper, {\em On the rational approximation of the sum of the reciprocals of the Fermat numbers} published in this same journal, we prove that for each integer $b\geq 2$ the irrationality exponent of $\sum_{n\geqslant 0} s_2(n)/b^n$ is equal to $2$.
\end{abstract}

\maketitle


\vspace{-.8cm}
\section{Introduction}

Let $\ga$ be a real number. The {\em irrationality exponent} $\mu(\ga)$ is defined as the supremum of the set of real numbers $\mu$ such that the inequality $$\left|\ga-\frac{p}{q}\right|<\frac{1}{q^\mu}$$ has infinitely many solutions $(p,q)\in\B{Z}\times\B{N}.$ For example, Liouville \cite{L1844} proved that $\mu(\sum_{n\geqslant 0}10^{-n!})=\infty$, and Roth \cite{R1955} showed that if $\ga$ is an irrational algebraic number, then $\mu(\ga)=2$. Note also that $\mu(\ga)\geqslant 2$ for all irrational $\ga$.

In this Addendum, we prove the following theorem.

\begin{theorem}\label{main} Let $s_2(n)$ be the sum of the binary digits of $n$. Then for each integer $b\geq 2$ we have $$\mu\left(\sum_{n\geqslant 0}\frac{s_2(n)}{b^n}\right)=2.$$
\end{theorem}

\noindent Transcendence of these numbers was proved by Toshimitsu \cite{T1998} using Mahler's method.

\section{Preliminaries}

As above, let $s_2(n)$ denote the sum of the binary digits of $n$, and set $\C{S}(x):=\sum_{n\geqslant 0} s_2(x)x^n.$ Note that for all $n\geqslant 0$, we have that both $s_2(2n)=s_2(n)$ and $s_2(2n+1)=s(n)+1.$ 
We prove our result by exploiting a connection between the sequences $\{s_2(n)\}_{n\geqslant 0}$ and $\{f(n)\}_{n\geqslant 1}$, where we define $f(n)$ by its generating series $$\C{F}(x):=\sum_{n\geqslant 1}f(n)x^n=\sum_{n=0}^\infty\frac{x^{2^n}}{1+x^{2^n}}.$$ The series $\C{F}(x)$ and its special values have been studied by many authors, including Golomb \cite{G1963}, Duverney \cite{D2001}, and Schwarz \cite{S1967}. This series is of special interest as $\C{F}(1/2)$ is the sum of the reciprocals of the Fermat numbers. Our interest here is tied to the following result.

\begin{proposition}[Coons \cite{C2012}]\label{IEF} Let $b\geqslant 2$ be a positive integer. Then $\mu(\C{F}(1/b))=2.$ 
\end{proposition}

To use Proposition \ref{IEF} to prove Theorem \ref{main}, we will use the relationship contained in the following lemma.

\begin{lemma} For all $n\geqslant 1$ we have $f(n)=s_2(n)-s_2(n-1)$.
\end{lemma}

\begin{proof} Note that the generating function for $\C{F}(x)$ implies that $f(n)$ is multiplicative, and on prime powers given by $$f(p^k)=\begin{cases} 1-k &\mbox{if $p=2$}\\ 1 &\mbox{if $p\neq 2$}.\end{cases}$$ Consider the function $v(n):=s_2(n)-s_2(n-1)$ for $n\geqslant 1$. 

To this end, note that if $n$ is odd, say $n=2k+1$, then $$v(2k+1)=s_2(2k+1)-s_2(2k)=s_2(k)+1-s_2(k)=1=f(2k+1).$$ If $n$ is even, say $n=2^k(2\ell+1)$, then \begin{align*} v(2^k(2\ell+1))&=s_2(2^k(2\ell+1))-s_2(2^k(2\ell+1)-1)\\
&=s_2(2\ell+1)-s_2(2^{k+1}\ell+2^{k}-1)\\ 
&=s_2(\ell)+1-s_2(2^{k+1}\ell)-s_2(2^{k}-1)\\
&=s_2(\ell)+1-s_2(\ell)-k\\
&=1-k.
\end{align*} From here it is easy to see that $v(n)$ is multiplicative and $f(p^k)=v(p^k)$ for all primes $p$ and integers $k\geqslant 1$. Thus $v(n)=f(n)$.
\end{proof}

\section{Proof of the main result}


\begin{proof}[Proof of Theorem \ref{main}] Using the above lemma and the fact that $s_2(0)=0$, we have that \begin{align*} \C{F}(x)=\sum_{n\geqslant 1}f(n)x^n&=\sum_{n\geqslant 1}(s_2(n)-s_2(n-1))x^n=\sum_{n\geqslant 1}s_2(n)x^n-\sum_{n\geqslant 1}s_2(n-1)x^n\\
&=\sum_{n\geqslant 0}s_2(n)x^n-x\sum_{n\geqslant 0}s_2(n)x^n=(1-x)\C{S}(x),\end{align*} so that $\C{S}(1/b)=\frac{b}{b-1}\cdot\C{F}(1/b)$. Since $\C{S}(1/b)$ is a (nonzero) rational multiple of $\C{F}(1/b)$, they have the same irrationality exponent. Appealing to Theorem \ref{IEF} proves the theorem.
\end{proof}

\noindent{\em Remark.} Since $\C{F}(1/b)$ is a (nonzero) rational multiple of $\C{S}(1/b)$, the transcendence of $\C{S}(1/b)$ as proved by Toshimitsu \cite{T1998} provides an alternative proof (to that of Duverney \cite{D2001}) of the transcendence of $\C{F}(1/b)$.

\bibliographystyle{amsplain}
\providecommand{\bysame}{\leavevmode\hbox to3em{\hrulefill}\thinspace}
\providecommand{\MR}{\relax\ifhmode\unskip\space\fi MR }
\providecommand{\MRhref}[2]{%
  \href{http://www.ams.org/mathscinet-getitem?mr=#1}{#2}
}
\providecommand{\href}[2]{#2}

\end{document}